\documentclass[12pt]{amsart}
\usepackage{amssymb}
\usepackage{amsfonts}

\usepackage[left=3cm,right=2cm,top=2cm,bottom=2cm]{geometry}
\newtheorem{ithm}{Theorem} 
\newtheorem{thm}{Theorem}[section]
\newtheorem{cor}[thm]{Corollary}
\newtheorem{lem}[thm]{Lemma}
\newtheorem{prop}[thm]{Proposition}
\theoremstyle{definition}

\newcommand\blank{\mathord{\hbox to 1.5ex{\hrulefill}}\,}
\DeclareMathOperator\GL{GL}
\DeclareMathOperator\G{G}

\DeclareMathOperator\E{E}
\DeclareMathOperator\M{M}

\DeclareMathOperator\sign{sign}
\newcommand\Sp{\operatorname{Sp}_{2n}}
\newcommand\Ep{\operatorname{Ep}_{2n}}
\newcommand\Min{\operatorname{Min}_n}

\newcommand\Z{\mathbb Z}
\let\ph\varphi
\let\eps\varepsilon
\renewcommand\le\leqslant
\renewcommand\ge\geqslant
\newcommand\T[1]{#1^{\scriptscriptstyle\mathsf T}}

\begin{document}

\title
{Subring subgroups in symplectic groups in characteristic 2}

\author{Anthony Bak}

\address
{Bielefeld University, Postfach 100131, 33501 Bielefeld, Germany}

\email{bak@mathematik.uni-bielefeld.de}

\author{Alexei Stepanov}
\address
{St.Petersburg State University\newline
and\newline
St.Petersburg Electrotechnical University
}

\email{stepanov239@gmail.com}

\thanks
{The work of the second named author under this publication is supported by
Russian Science Foundation, grant N.14-11-00297}

\begin{abstract}
In 2012 the second author obtained a description of the lattice of subgroups of a Chevalley group
$G(\Phi,A)$, containing the elementary subgroup $E(\Phi,K)$ over a subring $K\subseteq A$ provided
$\Phi=B_n,$ $C_n$ or $F_4$, $n\ge2$, and $2$ is invertible in $K$.
It turns out that this lattice is a disjoint union of ``sandwiches''
parameterized by subrings $R$ such that $K\subseteq R\subseteq A$.

In the current article a similar result is proved for $\Phi=C_n$, $n\ge3$,
and $2=0$ in $K$.
In this setting one has to introduce more sandwiches, namely, sandwiches which are parameterized
by form rings $(R,\Lambda)$ such that $K\subseteq\Lambda\subseteq R\subseteq A$.
The result we get generalizes Ya.\,N.\,Nuzhin's theorem of
2013 concerning the root systems $\Phi=B_n,$ $C_n$, $n\ge3$, where the same description
of the subgroup lattice is obtained, but under the condition that $A$ is an algebraic extension
of a field~$K$.
\end{abstract}

\maketitle

\section*{Introduction}

Throughout this paper $K$, $R$ and $A$ will denote commutative rings.
Let $G=\G_P(\Phi,\blank)$ denote a Chevalley--Demazure group scheme
with a reduced irreducible root system $\Phi$ and weight lattice $P$.
If the weight lattice $P$ is not important, we leave it out of the notation.
Denote by $E(A)=\E_P(\Phi,A)$ the elementary subgroup
of $G(A)$, i.\,e. the subgroup generated by all elementary root
unipotent elements $x_\alpha(t)$, \ $\alpha\in\Phi$, \ $t\in A$.
Let $K$ be a subring of $A$. We study  the lattice
$\mathcal L=L\bigl(E(K),G(A)\bigr)$ of subgroups of $G(A)$,
containing $E(K)$.

The standard description of $\mathcal L$ is called a \emph{sandwich classification
theorem}. It states that for each $H\in\mathcal L$ there
exists a unique subring $R$ between $K$ and $A$ such that $H$
lies between $E(R)$ and its normalizer $N_A(R)$ in $G(A)$.
The lattice $L\bigl(E(R),N_A(R)\bigr)$ of all subgroups of $G(A)$ that lie between $E(R)$ and
$N_A(R)$ is called a \emph{standard sandwich}.
Thus, the sandwich classification theorem holds iff $\mathcal L$ is the disjoint union of all standard
sandwiches. In~\cite{StepStandard} the second author proved the sandwich classification
theorem provided that $\Phi$ is doubly laced and $2$ is invertible in $K$.

\emph{In this article we consider the symplectic case of rank $n\ge3$ with $2=0$ in $A$,
in particular, we always assume that $\Phi=C_n$}.
In this situation we show that the sandwich classification theorem, as formulated above, does not hold.
Let $R$ be a subring of $A$,
containing $K$. Recall~\cite{BakBook} that an additive subgroup $\Lambda$ of $R$ is called a
(symplectic) \emph{form parameter} if it contains $2R$ and is closed under multiplication by $\xi^2$
for all $\xi\in R$.
If $2=0$, the set $R^2=\{\xi^2\mid \xi\in R\}$ is a subring of $R$,
and a form parameter $\Lambda$ in $R$ is just an $R^2$-submodule of $R$. Let $\Ep(R,\Lambda)$
denote the subgroup of $\Sp(R)$ generated by all root unipotents $x_\alpha(\mu)$  and $x_\beta(\lambda)$,
where $\alpha$ is a short root and $\mu \in R$ and $\beta$ is a long root and $\lambda\in\Lambda$. Suppose now
that $\Lambda\supseteq K$. Then clearly $\Ep(R,\Lambda)\ge\Ep(K)$. But one can check
that if $\Lambda\ne R$ then $\Ep(R,\Lambda)$ is not contained in any standard sandwich
$L\bigl(\Ep(R'),N_A(R')\bigr)$ such that $K \subseteq R'$. This shows that  $\mathcal L$
cannot be a union of standard
sandwiches, unless we enlarge our standard sandwiches or introduce additional ones. It turns out that the
latter approach is the correct one. It is analogous to that followed by the first author
in~\cite{BakThesis} and by E.\,Abe and K.Suzuki in~\cite{AbeSuzuki,AbeNormal} in order
to provide enough  sandwiches to classify subgroups of $\Sp(R)$, which are normalized by $\Ep(R)$.

If $A$ and $K$ are field and $A$ is algebraic over $K$, then the sandwich classification
theorem was obtained by Ya.\,N.\,Nuzhin in~\cite{Nuzhin,Nuzhin13}.
In~\cite{Nuzhin13} he considered the case of doubly laced root systems
in characteristic $2$ and $G_2$ in characteristics $2$ and $3$. In all these cases sandwiches
were parameterized by pairs; each pair consists of a subring and an additive subgroup,
satisfying certain properties.

Let $F$ be a field of characteristic $2$. Then there are injections
$\Sp(F)\rightarrowtail\operatorname{SO}_{2n+1}(F)\rightarrowtail\Sp(F)$, which turn to isomorphisms
if $F$ is perfect, see~\cite{Steinberg} Theorem~28, Example~(a) after this theorem, and Remark before Theorem~29.
Note that in this case the Chevalley groups $\operatorname{G}_P(C_n,F)$ and $\operatorname{G}_P(B_n,F)$
do not depend on the weight lattice $P$ (\cite{Steinberg}, Exercise after Corollary~5 to Theorem~$4'$).
Therefore, the lattice $L\bigl(\operatorname{E}(\Phi,K),\operatorname{G}(\Phi,F)\bigr)$ embeds to the lattice
$L\bigl(\operatorname{Ep}_{2n}(\mathbb F_2),\Sp(F)\bigr)$ for $\Phi=B_n,C_n$ and a subring $K$ of $F$,
and its description follows from the main result of the current article.

Actually, the injections above are constructed in~\cite{Steinberg} on the level of Steinberg groups.
Since over a field $K_2(\Phi,F)$ is generated by Steinberg symbols, they induce the injections of Chevalley
groups. Over a ring $A$ there is no appropriate description of $K_2$ available, therefore we can not
use the above arguments. Moreover, in the ring case the group $\operatorname{G}_P(\Phi,A)$
depends on $P$. By these reasons, the adjoint group of type $C_n$ and groups of type $B_n$ over rings will be
considered in a subsequent article.

We state now our main result. Following~\cite{BakBook}, we call a pair $(R,\Lambda)$ consisting
of a ring $R$ and a form parameter $\Lambda$ in $R$ a \emph{form ring}.

\begin{ithm}\label{MAIN}
Let $A$ denote a (commutative) ring such $2=0$ in $A$.  Let $K$ be a subring of $A$. If $H$ is a subgroup
of\/ $\Sp(A)$ containing\/ $\Ep(K)$, $n\ge 3$,  then there is a unique form ring $(R,\Lambda)$ such that
$K \subseteq \Lambda  \subseteq R \subseteq A$ and

$$
\Ep(R,\Lambda)\le H\le N_A(R,\Lambda),
$$

\noindent
where $N_A(R,\Lambda)$ is the normalizer of\/ $\Ep(R,\Lambda)$ in\/ $\Sp(A)$.
\end{ithm}

Let $\mathcal L(R,\Lambda)=L\bigl(\Ep(R,\Lambda), N_A(R,\Lambda)\bigr)$ denote the lattice of all
subgroups $H$ of $\Sp(A)$ such that $\Ep(R,\Lambda)\le H\le N_A(R,\Lambda)$.
From now on, $\mathcal L(R,\Lambda)$ is what we shall mean by a \emph{standard sandwich}.
In view of Theorem~\ref{MAIN} it is natural to study the lattice structure of $\mathcal L(R,\Lambda)$.
By definition $\Ep(R,\Lambda)$ is normal in $N_A(R,\Lambda)$. Therefore the lattice structure of
$\mathcal L(R,\Lambda)$ is the same as that of the quotient group $N_A(R,\Lambda)/\Ep(R,\Lambda)$.
The lattice $\mathcal L(R,\Lambda)$
contains an important subgroup $\Sp(R,\Lambda)$, which is called a Bak symplectic group,
whose definition will be recalled in Section~1. A group will be called \emph{quasi-nilpotent} if it is a
direct limit of nilpotent subgroups, i.\,e. the group has a directed system of nilpotent subgroups
whose colimit is the group itself.

In the following result we use the notion of Bass--Serre
dimension of a ring introduced by the first author in~\cite[\S~4]{BakNonabelian}.
Recall that the Krull (or combinatorial) dimension of a topological space is the supremum
of the lengths of proper chains of nonempty closed irreducible subsets.
Bass--Serre dimension $d=\operatorname{BS-dim} R$ of a ring $R$ is the smallest integer such that the
maximal spectrum $\operatorname{Max}R$ is a union of a finite number of irreducible Noetherian
subspaces of Krull dimension not greater than $d$. If there is no such integer, then
$\operatorname{BS-dim} R=\infty$.

\begin{ithm}\label{nilpotent}
Let $A$ denote a (commutative) ring and let $R$ be a subring of $A$
and $(R,\Lambda)$ a form ring. Suppose $n\ge 2$.
\begin{enumerate}
\item $\Sp(R,\Lambda) = \Sp(R) \cap N_A(R,\Lambda)$.
\item $\Sp(R,\Lambda)$ is normal in $N_A(R,\Lambda)$.
\item $N_A(R,\Lambda)/\Sp(R,\Lambda)$ is abelian.
\item $\Sp(R,\Lambda)/\Ep(R,\Lambda)$ is quasi-nilpotent.
\item Let $R_0$ denotes the subring of $R$, generated by all elements $\xi^2$ such that
$\xi\in R$. If Bass--Serre dimension of $R_0$ is finite, then $\Sp(R,\Lambda)/\Ep(R,\Lambda)$ is nilpotent.
\end{enumerate}
In particular, the sandwich quotient group $N_A(R,\Lambda)/\Ep(R,\Lambda)$ is quasi-nilpotent
by abelian or nilpotent by abelian if $\operatorname{BS-dim} R_0<\infty$.
\end{ithm}

Although Theorem~\ref{MAIN} is proven under the assumptions $2=0$ and $n\ge3$, we do not invoke these
assumptions until the proof of the theorem in Section~6. In particular, Theorem~\ref{nilpotent} is
proven without these assumptions. Of course, we always assume that $n\ge2$.

\subsection*{Notation}
Let $H$ be a group.
For two elements $x,y\in H$ we write $[x,y]=xyx^{-1}y^{-1}$ for their
commutator and $x^y=y^{-1}xy$ for the $y$-conjugate of $x$.
For subgroups $X,Y\le H$ we let $X^Y$ denote the normal closure of
$X$ in the subgroup generated by $X$ and $Y$,
while $[X,Y]$ stands for the mixed commutator group generated by $X$ and $Y$.
By definition it is the group generated by all commutators $[x,y]$ such that
$x\in X$ and $y\in Y$.
The commutator subgroup $[X,X]$ of $X$ will be also denoted by $D(X)$ and we
set $D^k(X)=\bigl[D^{k-1}(X),D^{k-1}(X)\bigr]$.
Recall that a group $X$ is called \emph{perfect} if $D(X)=X$.

The identity matrix is denoted by $e$ as well as the identity element of a Chevalley group.
We denote by $e_{ij}$ the matrix with $1$ in position~$(ij)$
and zeroes elsewhere. The entries of a matrix $g$ are denoted by $g_{ij}$.
For the entries of the inverse matrix we use abbreviation $(g^{-1})_{ij}=g'_{ij}$.
The transpose of the matrix $g$ is denoted by $\T g$, thus $(\T g)_{ij}=g_{ji}$.


\section{The symplectic group}\label{Sp}

The symplectic group $\Sp(R)$, its elementary root unipotent elements, and its elementary
subgroup $\Ep(R)$ will be recalled below. The groups $\Sp(R,\Lambda)$ and their elementary
subgroups $\Ep(R,\Lambda)$ will be defined in Section~\ref{Sp(R,Lambda)}.

Since the groups $\Sp(R,\Lambda)$ are not in general algebraic
(in fact, $\Sp(R,\Lambda)$ is algebraic iff $\Lambda=R$ or $\Lambda=\{0\}$),
it is convenient to work with the standard matrix representation of $\Sp(R)$.
On the other hand, we want to use the notions of parabolic subgroup, unipotent radical, etc.\
from the theory of algebraic groups. Thus, to simplify the exposition, we define below these notions
directly in terms of the matrix representation we use.

Following Bourbaki~\cite{Bourbaki4-6} we view the root system $C_n$ and its set of fundamental
roots $\Pi$ in the following way.
Let $V_n$ denote $n$-dimensional euclidean space with the orthonormal basis $\eps_1,\dots,\eps_n$.
Let

\begin{align*}
C_n&=\{\pm\eps_i\pm\eps_j\mid 1\le i<j\le n\}\cup\{\pm2\eps_k\mid 1\le k\le n\},\\
\Pi&=\{\alpha_i=\eps_i-\eps_{i+1},\text{ where }i=1,\dots,n-1,\ \alpha_{n}=2\eps_n\}.
\end{align*}

\noindent
The elements $\pm\eps_i\pm\eps_j$ are called \emph{short roots} and the elements $\pm2\eps_k$ \emph{long roots}.

From the perspective of algebraic groups, we want the standard Borel subgroup of $\Sp(R)$ to be the
group of all upper triangular matrices in $\Sp(R)$.
Accordingly, we take the following matrix description of $\Sp(R)$.
Let $J$ denote the $n\times n$-matrix with 1 in each antidiagonal position and zeroes elsewhere. Let
$F=\left(\begin{smallmatrix}0&J\\ -J& 0\end{smallmatrix}\right)$.
Then, the group $\Sp(R)$ is the
subgroup of $\GL_{2n}(R)$ consisting of all matrices which preserve the bilinear
form whose matrix is $F$. In other words,
$$
\Sp(R)=\{g\in\GL_{2n}(R)\mid \T{g}Fg=F\}.
$$

Let $I=(1,\dots,n,-n,\dots,-1)$ denote the linearly ordered set whose linear ordering is obtained
by reading from left to right.
We enumerate the rows and columns of the matrices of
$\GL_{2n}(R)$ by indexes from $I$. Thus the position of a coordinate of a matrix in
$\GL_{2n}(R)$ is denoted by a pair $(i,j)\in I\times I$.

In the language of algebraic groups the set of all diagonal matrices in $\Sp(R)$ is a \emph{maximal torus}.
Let $e$ denote the $2n\times 2n$ identity matrix.
The following matrices are \emph{elementary root unipotent elements} of $\Sp(R)$
with respect to the torus above:

\begin{align*}
x_{\eps_i-\eps_j}(\xi)&=T_{ij}(\xi)=T_{-j,-i}(-\xi)=e+\xi e_{ij}-\xi e_{-j,-i},\\
x_{\eps_i+\eps_j}(\xi)&=T_{i,-j}(\xi)=T_{j,-i}(\xi)=e+\xi e_{i,-j}+\xi e_{j,-i},\\
x_{-\eps_i-\eps_j}(\xi)&=T_{-i,j}(\xi)=T_{-j,i}(\xi)=e+\xi e_{-i,j}+\xi e_{-j,i},\\
x_{2\eps_k}(\xi)&=T_{k,-k}(\xi)=e+\xi e_{k,-k},\\
x_{-2\eps_k}(\xi)&=T_{-k,k}(\xi)=e+\xi e_{-k,k},\\
&\text{where }\xi\in R,\ 1\le i,j,k\le n,\ i\ne j.\\
\end{align*}

Note that the subscripts $(i,j)$, $(i,j)$, $(i,-j)$, $(j,-i)$, $(-j,-i)$, $(k,-k)$, and $(-k,k)$
on the $T$ above all belong to the set
$\tilde C_n=I\times I\setminus \{(k,k)\mid k\in I\}$ of nondiagonal positions of a matrix
from $\Sp(R)$ and exhaust $\tilde C_n$.

There is a surjective map $p:\tilde C_n\to C_n$ defined by the following rule.

\begin{align*}
p(i,j)&=p(-j,-i)=\eps_i-\eps_j;\\
p(i,-j)&=p(j,-i)=\eps_i+\eps_j;\\
p(-i,j)&=p(-j,i)=-\eps_i-\eps_j;\\
p(k,-k)&=2\eps_k;\\
p(-k,k)&=-2\eps_k;\\
\text{where }&\xi\in R,\ 1\le i,j,k\le n,\ i\ne j.\\
\end{align*}

\noindent
With this notation the correspondence between elementary symplectic transvections $T_{ij}(\xi)$
and root elements $x_\alpha(\xi)$ looks as follows.
$$
T_{ij}(\xi)=x_{p(i,j)}(-\sign(ij)\xi)\qquad\text{for all } i\ne j\in I.
$$
Note that $p$ maps symmetric (with respect to the antidiagonal) positions to the same root.
Therefore, for $(ij)\in I$ and $\xi\in R$ we have $T_{i,j}(\xi) = T_{-j,-i}(-\sign(ij)\xi)$.


The root subgroup scheme $X_\alpha$ is defined by $X_\alpha(R)=\{x_\alpha(\xi)\mid\xi\in R\}$.
The scheme $X_\alpha$ is naturally isomorphic to $\mathbb G_a$, i.\,e.
$x_\alpha(\xi)x_\alpha(\mu)=x_\alpha(\xi+\mu)$ for all $\xi,\mu\in R$.
The following commutator formulas are
well known in matrix language, cf.~\cite[\S~3]{BakVavHyp1}. They are special cases of the Chevalley
commutator formula in the algebraic group theory.

\begin{align*}
[x_\alpha(\lambda),x_\beta(\mu)]&=x_{\alpha+\beta}(\pm\lambda\mu),
\text{ if }\alpha+\beta\in\Phi,\,\widehat{\alpha\beta}=2\pi/3;\\
[x_\alpha(\lambda),x_\beta(\mu)]&=x_{\alpha+\beta}(\pm2\lambda\mu),
\text{ if }\alpha+\beta\in\Phi,\,\widehat{\alpha\beta}=\pi/2;\\
[x_\alpha(\lambda),x_\beta(\mu)]&=x_{\alpha+\beta}(\pm\lambda\mu)x_{\alpha+2\beta}(\pm\lambda\mu^2),
\text{ if }\alpha+\beta,\alpha+2\beta\in\Phi,\,\widehat{\alpha\beta}=3\pi/4;\\
[x_\alpha(\lambda),x_\beta(\mu)]&=e,
\text{ if }\alpha+\beta\notin\Phi\cup\{0\}
\end{align*}

In our proofs we make frequent use of the \emph{parabolic subgroup}~$P_1$ of $\Sp$.
In the matrix language above it is defined as follows:

$$
P_1(R)=\{g\in\Sp{R}\mid g_{i1}=g_{-1-i}=0\ \forall i\ne1\}
$$

\noindent
The definition above of $\Sp(R)$ shows that $g_{11}=g^{-1}_{-1,-1}$ for
any matrix $g\in P_1(R)$. The \textit{unipotent radical} $U_1$ of $P_1$ is the subgroup
generated by all root subgroups $T_{1i}$ such that $i\ne 1$.
The \emph{Levi subgroup} $L_1(R)$ of $P_1(R)$ consists of all $g\in P_1(R)$
such that $g_{1i}=g_{-i,-1}=0$ for all $i\ne 1$. As a group scheme it is isomorphic
to $\mathbb G_m\times\operatorname{Sp}_{2n-2}$.

\section{Bak symplectic groups}\label{Sp(R,Lambda)}
The Bak symplectic group $\Sp(R,\Lambda)$ is the particular case of the Bak general
unitary group, where the involution is trivial and the symmetry $\lambda=-1$.
The main references for the definition and the structure of the general unitary group
is the book~\cite{BakBook} and the paper~\cite{BakVavHyp1} by N.\,Vavilov and the first author.
In this section we recall definitions and simple properties to be used in the sequel.

Let $R$ be a commutative ring.
An additive subgroup $\Lambda$ of $R$ is called a \emph{symplectic form parameter}
in $R$, if it contains $2R$ and is closed under multiplication by squares, i.\,e. $\mu^2\lambda\in\Lambda$
for all $\mu\in R$ and $\lambda\in\Lambda$.
Define $\Sp(R,\Lambda)$ as the subgroup of $\Sp(R)$ consisting of all matrices preserving
the quadratic form which takes values in $R/\Lambda$ and is defined by the matrix
$\left(\begin{smallmatrix}0 & J\\ 0 & 0\end{smallmatrix}\right)$.

Let $*$ denote the involution on the matrix ring $\M_n(R)$ given by the formula
$a^*=J\T aJ$. Note that this is the reflection of a matrix with respect to the antidiagonal.
Define
$$
\M_n(R,\Lambda)=\{a\in\M_n(R)\mid a=a^*,\,a_{n-k+1\,k}\in\Lambda\text{ for all }k=1,\dots,n\}.
$$
Write a matrix of degree $2n$ in the block form
$\left(\begin{smallmatrix}a & b\\ c & d\end{smallmatrix}\right)$, where $a,b,c,d\in\M_n(R)$.
It follows from~\cite[Lemma~2.2]{BakVavHyp1} and its proof that under the notation above
we have the following formula.

\begin{lem}\label{BakSp}
$
\Sp(R,\Lambda)=\left\{\begin{pmatrix}a & b\\ c & d\end{pmatrix}\mid
a^*d-c^*b=e\text{ and }c^*a,d^*b\in\M_n(R,\Lambda)\right\}.
$
\end{lem}

It is easy to check that $\M_n(R,\Lambda)$ is a form parameter in the ring $\M_n(R)$ with
involution $*$, corresponding to the symmetry $\lambda=-1$.
The minimal form parameter in $\M_n(R)$ with the same involution and symmetry is
$\M_n(R,2R)$; it is denoted by $\Min(R)$.
Let $\bar\M_n(R,\Lambda)$ denote the additive group $\M_n(R,\Lambda)/\Min(R)$.

\begin{lem}\label{barM}
The group $\bar\M_n(R,\Lambda)$
has a natural structure of a left $M_n(R)$-module under the operation
$a\circ\bar b=aba^*\mod\Min(R)$, where $a\in\M_n(R)$, \ $\bar b\in\bar\M_n(R,\Lambda)$,
and $b$ is a preimage of $\bar b$ in $\M_n(R,\Lambda)$.
\end{lem}

By abuse of notation for $a\in\M_n(R)$ and $b\in\M_n(R,\Lambda)$ we shall write $a\circ b$ instead of
$a\circ(b+\Min(R))$.

It is easy to check that an elementary root unipotent element $x_\alpha(\xi)\in\Sp(R)$ belongs to
$\Sp(R,\Lambda)$
if and only if $\alpha$ is a short root or $\xi\in\Lambda$. Denote by $\Ep(R,\Lambda)$ the group
generated by all such elements:
$$
\Ep(R,\Lambda)=\langle x_\alpha(\xi)\mid
\alpha\in C_n^{short}\,\&\,\xi\in R\,\lor\alpha\in C_n^{long}\,\&\,\xi\in\Lambda\rangle.
$$

\section{Subgroups generated by elementary root unipotents}\label{bottom}

In this section we assume that $n\ge3$.
Let $H$ be a subgroup of $\Sp(A)$, containing $\Ep(K)$.
The following lemma shows that we can uncouple a short root element from
a long one inside $H$.

\begin{lem}\label{RightSide}
Let $\alpha,\beta\in C_n$ be a short and a long root, respectively, such that
$\alpha+\beta$ is not a root.
Let $g=x_{\alpha}(\mu)x_{\beta}(\lambda)$, where $\lambda,\mu\in A$.
If\/ $\Ep(K)^g\le H$ (e.\,g. $g\in H$), then each factor of $g$ belongs to $H$.
\end{lem}

\begin{proof}
Since $n\ge3$, there exists a short root $\gamma$ such that $\gamma+\alpha$
is a short root and $\gamma+\beta$ is not a root. Then, $X_\beta(A)$ commutes with
$X_\alpha(A)$ and $X_\gamma(A)$, hence $[g,x_\gamma(1)]=x_{\alpha+\gamma}(\pm\mu)\in H$.
Conjugating this element by an appropriate element from the Weyl group over $K$
and taking the inverse if necessary, we get $x_\alpha(\mu)\in H$. It follows that
$\Ep(K)^{x_{\beta}(\lambda)}\le H$.

Now, take a short root $\delta$ such that $\beta+\delta$ is a root. Then,
$[x_\delta(1),x_\beta(\lambda)]=
x_{\delta+\beta}(\pm\lambda)x_{2\delta+\beta}(\pm\lambda)\in H$.
Notice that the root $\delta+\beta$ is short whereas $2\delta+\beta$ is long.
As in the first paragraph of the proof one concludes that
$x_{\delta+\beta}(\pm\lambda)\in H$, hence $x_{2\delta+\beta}(\pm\lambda)\in H$.
Again, using the action of the Weyl group and taking inverse if necessary, one
shows that $x_{\beta}(\lambda)\in H$.
\end{proof}

Put $P_\alpha(H)=\{t\in A\,|\,x_\alpha(t)\in H\}$.
Since the Weyl group acts transitively on the set of roots of the same length,
it is easy to see that $P_\alpha(H)=P_\beta(H)$ if $|\alpha|=|\beta|$.
Let $R=R_H=P_\alpha(H)$ for any short root $\alpha$, and let
$\Lambda=\Lambda_H=P_\beta(H)$ for any long root $\beta$.

\begin{lem}\label{FormIdeal}
With the above notation $(R,\Lambda)$ is a form ring and
$K\subseteq \Lambda\subseteq R\subseteq A$.
\end{lem}

\begin{proof}
Clearly, $P_\alpha(H)$ is an additive subgroup of $A$.
Since $n\ge3$, there are two short roots $\alpha,\alpha'$ such that
$\alpha+\alpha'$ also is short. The commutator formula
$$
[x_\alpha(\lambda),x_{\alpha'}(\mu)]=x_{\alpha+\alpha'}(\pm\lambda\mu)
$$
shows that $R$ is a ring.

Now, let $\alpha,\alpha'$ be short roots which are orthogonal in $V_n$ and such that
$\beta=\alpha+\alpha'\in\Phi$ is a long root. Then
$$
[x_\alpha(\mu),x_{\alpha'}(1)]=x_{\beta}(\pm2\mu).
$$
If $\mu\in R$, then this element belongs to $H$. This proves that
$2R\subseteq\Lambda$. Finally, we show that $\Lambda$ is closed under multiplication
by squares in $R$. Let
$$
g=[x_\beta(\lambda),x_{-\alpha}(\mu)]=
x_{\alpha'}(\pm\lambda\mu)x_{\beta-2\alpha}(\pm\lambda\mu^2)
$$
If $\lambda\in\Lambda$ and $\mu=1$, then $g\in H$, and Lemma~\ref{RightSide} shows
that $x_{\alpha'}(\pm\lambda)\in H$. Therefore, $\Lambda\subseteq R$.
On the other hand, if $\mu\in R$ and $\lambda\in\Lambda$, then $g$ also lies in $H$,
and by Lemma~\ref{RightSide} we have $x_{\beta-2\alpha}(\pm\lambda\mu^2)\in H$.
This shows that $\Lambda$ is stable under multiplication by squares of elements of $R$.
\end{proof}

If $H$ is a subgroup of $\Sp(A)$ containing $\Ep(K)$,
then $(R_H,\Lambda_H)$ is called the \textit{form ring associated with} $H$.

\section{The normalizer}\label{top}

Let $(R,\Lambda)$ be a form subring of a ring $A$ such that $1\in\Lambda$.
In this section we develop properties of the normalizer
$N_A(R,\Lambda)$ of the group $\Ep(R,\Lambda)$ in $\Sp(A)$ and prove Theorem~\ref{nilpotent}.
Here we do not assume that $n\ge3$.
By a result of Bak and Vavilov~\cite[Theorem~1.1]{BakVavNormal}
we know that $\Ep(R,\Lambda)$ is normal in
$\Sp(R,\Lambda)$, thus $\Sp(R,\Lambda)\le N_A(R,\Lambda)$.
First, we show that the quotient $N_A(R,\Lambda)/\Sp(R,\Lambda)$ is abelian.

\begin{lem}\label{NormalizerElem}
Let $g\in\Sp(A)$.
If\/ $\Ep(R,\Lambda)^g\le \GL_{2n}(R)$ then
$g_{ij}g_{kl}\in R$ for all $i,j,k,l\in I$.
\end{lem}

\begin{proof}
To begin we express a matrix unit $e_{jk}$ as a linear combination $\sum\xi_m a^{(m)}$
for some  $\xi_m\in R$ and $a^{(m)}\in\Ep(R,\Lambda)$. Note that the set of such linear combinations
is closed under multiplication.
Suppose $j\ne k$ and let $i\ne\pm j,\pm k$. Then $e_{jk}=(T_{ji}(1)-e)(T_{ik}(1)-e)$ and
$e_{jj}=e_{jk}e_{kj}$. It follows that if $\Ep(R,\Lambda)^g\le \GL_{2n}(R)$, then
$g^{-1}e_{jk}g\in\M_n(R)$. Thus $g'_{ij}g_{kl}\in R$ for all $i,j,k,l\in I$
(recall that $g'_{ij}=(g^{-1})_{ij}$).
The conclusion of the lemma follows since the entries of $g^{-1}$ coincide with the entries of
$g$ up to sign and a permutation.
\end{proof}

The next proposition describes the normalizer in the case $A=R$.

\begin{prop}\label{Normalizer for R=A}
$N_R(R,\Lambda)=\Sp(R,\Lambda)$.
\end{prop}

\begin{proof}
Let $g=\left(\begin{smallmatrix}a&b\\c&d\end{smallmatrix}\right)\in N_R(R,\Lambda)$.
Since $g\in\Sp(R)$, then $d^*a-b^*c=e$.
We have to prove that $c^*a,\,d^*b\in M_n(R,\Lambda)$.
Since
$h=\left(\begin{smallmatrix}e&J\\0&e\end{smallmatrix}\right)\in\Ep(R,\Lambda)$,
we have
$$
f=ghg^{-1}=\begin{pmatrix}e-aJc^* & aJa^*  \\
                        -cJc^*  & e+cJa^*\end{pmatrix}\in\Ep(R,\Lambda)
$$

It follows that the matrices $-(cJc^*)^*(e-aJc^*)$ and $(e+cJa^*)^*(aJa^*)$ belong
to $M_n(R,\Lambda)$. Since $cJc^*$ and $aJa^*$ are automatically in
$M_n(R,\Lambda)$, we have that $cJc^*aJc^*,\,aJc^*aJa^*\in M_n(R,\Lambda)$.
Modulo $\Min(R)$ we can write
$(cJ)\circ (c^*a),\,(aJ)\circ (c^*a)\in \bar M_n(R,\Lambda)$.
By Lemma~\ref{barM} $\bar M_n(R,\Lambda)$ is an $M_n(R)$ module, therefore
\begin{multline*}
c^*a+\Min(R)=\bigl(J(d^*a-b^*c)J\bigr)\circ (c^*a)=\\
(Jd^*)\circ\bigl((aJ)\circ(c^*a)\bigr)-
(Jb^*)\circ\bigl((cJ)\circ(c^*a)\bigr)\in \bar M_n(R,\Lambda).
\end{multline*}
The proof that $d^*b\in M_n(R,\Lambda)$ is essentially the same.
\end{proof}

\begin{cor}\label{NormComm}
$[N_A(R,\Lambda),N_A(R,\Lambda)]\le\Sp(R,\Lambda)$.
\end{cor}

\begin{proof}
If $g,h\in N_A(R,\Lambda)$ and $f=[g,h]$, then by Lemma~\ref{NormalizerElem}
for all indexes $p,q$ we have
$$
f_{pq}=\sum\limits_{i,j,k=1}^n g'_{pi}h'_{ij}g_{jk}h_{kq}=
\sum\limits_{i,j,k=1}^n (g'_{pi}g_{jk})(h'_{ij}h_{kq})\in R
$$

\noindent
Therefore,
$$
f\in\Sp(R)\cap N_A(R,\Lambda)=N_R(R,\Lambda)=\Sp(R,\Lambda).
$$
\end{proof}

The first 3 items of Theorem~\ref{nilpotent} are already proved.
Our next goal is to show that the main theorem of~\cite{HazDim} implies that
$\Sp(R,\Lambda)/\Ep(R,\Lambda)$ is nilpotent, provided that $R_0$ has finite
Bass--Serre dimension. This will imply the rest of Theorem~\ref{nilpotent}.
Recall that $R_0$ denotes the subring of $R$ generated by all elements $\xi^2$
such that $\xi\in R$.

\begin{lem}\label{semilocal}
If $R_0$ is semilocal and $1\in\Lambda$, then $\Sp(R,\Lambda)=\Ep(R,\Lambda)$.
\end{lem}

\begin{proof}
Since $R$ is integral over $R_0$, it is a direct limit of $R_0$-subalgebras $R'\subseteq R$
such that $R'$ is module finite and integral over $R_0$. By the first theorem of Cohen--Seidenberg,
each $R'$ is semilocal. Thus by~\cite[Lemma~4]{Bak75} $\Sp(R',\Lambda\cap R')=\Ep(R',\Lambda\cap R')$
for each $R'$. Since $\Sp$ and $\Ep$ commute with direct limits, it follows that
$\Sp(R,\Lambda)=\Ep(R,\Lambda)$.
\end{proof}

\begin{lem}\label{FinGen}
If $R$ is a finitely generated $\Z$-algebra then so is $R_0$.
\end{lem}

\begin{proof}
Since $2\xi=(\xi+1)^2-\xi^2-1$, we have $2R\subseteq R_0$.
Let $S$ be a finite set of generators for $R$ as a $\Z$-algebra.
We set $S_0=\{s^2\mid s\in S\}\cup\{2\prod_{s\in S'}s\mid S'\subseteq S\}$
and denote by $R_1$ the $\Z$-subalgebra generated by $S_0$.
Clearly, $2R\subseteq R_1$. The map $\xi\mapsto\xi^2$ is a ring epimorphism
$R/2R\to R_0/2R$, therefore $R_0/2R$ is generated by the images of generators of
$R$. It follows that $R_0/2R=R_1/2R$, hence $R_0=R_1$ is finitely generated.
\end{proof}

\begin{proof}[Proof of Theorem~\ref{nilpotent}]
Since $N_R(R,\Lambda)=N_A(R,\Lambda)\cap\Sp(R)$, the first statement follows from
Proposition~\ref{Normalizer for R=A}, the second is a particular case of~\cite[Theorem~1.1]{BakVavNormal},
and the third one coincides with Corollary~\ref{NormComm}.

Recall that $\Sp^0(R,\Lambda)=\cap_\ph\operatorname{Ker}\ph$, where $\ph$ ranges over all group homomorphisms
$\Sp(R,\Lambda)\to\Sp(\tilde R,\tilde \Lambda')/\Ep(\tilde R,\tilde \Lambda)$ induced by form ring morphisms
$(R,\Lambda)\to(\tilde R,\tilde\Lambda)$ such that $\tilde R_0$ is semilocal. By Lemma~\ref{semilocal}
$\Sp^0(R,\Lambda)=\Sp(R,\Lambda)$.

Since $R$ is integral over $R_0$, it is a direct limit of  $R_0$-subalgebras $R'\subseteq R$
such that $R'$ is module finite over $R_0$.
By~\cite{HazDim}[Theorem~3.10]
the group $\Sp^0(R',\Lambda\cap R')/\Ep(R',\Lambda\cap R')$ is nilpotent, provided that $R_0$ has
finite Bass--Serre dimension. Since $\Sp$ and $\Ep$ commute with direct limits, this proves (5).

Any commutative ring is a direct limit of finitely generated $\Z$-algebras. Let
$R=\injlim R^{(i)}$, where $i$ ranges over some index set $I$ and each $R^{(i)}$
is a finitely generated $\Z$-algebra. Then $R_0=\injlim R^{(i)}_0$
and $\Sp(R,\Lambda\cap R)/\Ep(R,\Lambda)=\injlim Sp(R^{(i)},\Lambda\cap R^{(i)})/\Ep(R^{(i)},\Lambda\cap R^{(i)})$.
By Lemma~\ref{FinGen} each $R^{(i)}_0$ is a finitely generated $\Z$-algebra. Therefore,
this ring has finite Krull dimension and hence finite Bass--Serre dimension. By~(5) each
group $Sp(R^{(i)},\Lambda\cap R^{(i)})/\Ep(R^{(i)},\Lambda\cap R^{(i)})$ is
nilpotent which implies item~(4).
\end{proof}

The following statement is crucial for the proof of our main theorem.
For Noetherian rings it is almost immediate consequence of Theorem~\ref{nilpotent}.

\begin{lem}\label{normalizer}
Let $g\in\Sp(A)$, $n\ge3$. If\/ $\Ep(R,\Lambda)^g\le N_A(R,\Lambda)$, then
$g\in N_A(R,\Lambda)$. Moreover, $\Ep(R,\Lambda)$ is a characteristic
subgroup of $N_A(R,\Lambda)$.
\end{lem}

\begin{proof}
Let $\theta$ be either an automorphism of $N_A(R,\Lambda)$ or an automorphism of $\Sp(A)$
such that $\Ep(R,\Lambda)^\theta\le N_A(R,\Lambda)$
(we denote by $h^\theta$ the image of an element $h\in N_A(R)$
under the action of $\theta$). Since $n\ge3$, the Chevalley commutator formula (see section~\ref{Sp})
implies that the group $\Ep(R,\Lambda)$ is perfect. By Corollary~\ref{NormComm} we have
\begin{multline*}
\Ep(R,\Lambda)^\theta=[\Ep(R,\Lambda),\Ep(R,\Lambda)]^\theta\le \\
[N_A(R,\Lambda),N_A(R,\Lambda)]\le\Sp(R,\Lambda).
\end{multline*}
We shall prove that $h^\theta\in \Ep(R,\Lambda)$ for any $h\in\Ep(R,\Lambda)$.
Write $h$ as a product of elementary root unipotents
$x_{\alpha_1}(s_1)\cdots x_{\alpha_m}(s_m)$. Let $R'$ denote the $\Z$-subalgebra of $R$ generated by
all $s_i$'s, and let $\Lambda'$ denote the form parameter of $R'$ generated by
those $s_j$ for which $\alpha_j$ is a long root.
Clearly $h\in\Ep(R',\Lambda')$ and $\Ep(R',\Lambda')$ is a finitely generated group.
Let $R''$ denote the $R'$-algebra generated by all entries of the matrices
$y^\theta$, where $y$ ranges over all generators of $\Ep(R',\Lambda')$.
Let $\Lambda''=\Lambda\cap R''$.
The inclusion $\Ep(R,\Lambda)^\theta\le\Sp(R,\Lambda)$ shows that $R''\subseteq R$.
Note that $\Ep(R',\Lambda')^\theta\le\Sp(R'',\Lambda'')$  by the choice of $R''$.

Since $R''$ is a finitely generated $R'$-algebra, it is
a finitely generated $\Z$-algebra. By Lemma~\ref{FinGen} $R''_0$ is a finitely generated
$\Z$-algebra. Therefore, it has finite Krull dimension and hence finite Bass--Serre dimension.
Thus by Theorem~\ref{nilpotent}(5), the $k$th commutator subgroup $D^k\Sp(R'',\Lambda'')$ equals to
$\Ep(R'',\Lambda'')$ for some positive integer $k$.
Now, since $\Ep(R',\Lambda')$ is perfect, it is equal to
$D^k\Ep(R',\Lambda')$.  It follows that
$$
\Ep(R',\Lambda')^\theta=D^k\Ep(R',\Lambda')^\theta\le
D^k \Sp(R'',\Lambda'')=\Ep(R'',\Lambda'').
$$
In particular, $h^\theta\in \Ep(R'',\Lambda'')\le \Ep(R,\Lambda)$.
Thus, $\Ep(R,\Lambda)$ is invariant under $\theta$.

If $\theta$ is an automorphism of $N_A(R)$ this means that
$\Ep(R,\Lambda)$ is a characteristic subgroup of $N_A(R)$.
If $\theta$ is an inner automorphism defined by $g\in G(A)$, then
the statement we proved is the first assertion of the lemma.
\end{proof}

The following straightforward corollary shows that the normalizers of
all subgroups of the sandwich $L\bigl(\Ep(R,\Lambda),N_A(R,\Lambda)\bigr)$
lie in that sandwich.

\begin{cor}
For any $H\le N_A(R,\Lambda)$ containing $\Ep(R,\Lambda)$ its normalizer
is contained in $N_A(R,\Lambda)$.
In particular, the group $N_A(R,\Lambda)$ is self normalizing.
\end{cor}

\section{Inside a parabolic subgroup}\label{parabolic}

Let $H$ be a subgroup of $\Sp(A)$, normalized by $\Ep(K)$.
Denote by $(R,\Lambda)$ the form ring, associated with $H$. In the proof of the
following lemma we keep using the ordering on the index set $I$, defined in
section~\ref{Sp}. For example, the product $\prod_{j=2}^{-1}$ means that
$j=2,\dots,n,-n,\dots,-1$. We assume that the order of factors agrees with the
ordering on $I$.

\begin{lem}\label{InUniRad}
If $g\in U_1(R)$ and $\Ep(K)^g\in H$, then $g\in\Ep(R,\Lambda)$.
\end{lem}

\begin{proof}
Let $g=\prod_{j=2}^{-1}T_{1j}(\mu_j)$.
We have to prove that $\mu_j\in R$ for any $j\ne -1$ and $\mu_{-1}\in\Lambda$.
Let $h$ be the smallest element from $I$ such that $\mu_j\ne0$.
We proceed by going down induction on $h$.
If $h=-1$, then $g$ consists of a single factor and
there is nothing to prove. If $h=-2$, then the result follows from
Lemma~\ref{RightSide}.

Now, let $h\le-2$. Denote by $i$ the successor of $h$ in $I$.
Then, $[T_{h,i}(-1),g]=T_{1,i}(\mu_h)\prod\limits_{j>i}T_{1,j}(\xi_h)$.
By induction hypothesis $\mu_h\in R$. The element $T_{1h}(-\mu_h)g$ satisfies the
conditions of the lemma. Again by induction hypothesis it belongs to
$\Ep(R,\Lambda)$. Thus, $g\in\Ep(R,\Lambda)$.
\end{proof}

It is known that in a Chevalley group the unipotent radicals of two opposite
standard parabolic subgroups span the elementary group (see e.\,g.~\cite[Lemma~2.1]{StepComput}).
The next lemma shows that this holds for the parabolic subgroup $P_1$ of $\Sp(R,\Lambda)$ as well.

\begin{lem}
The set
$$
\{T_{1i}(\mu),\,T_{i1}(\mu),\,T_{\pm1\,\mp1}(\lambda)\mid i\ne\pm1,\,\mu\in R,\,\lambda\in\Lambda\}
$$
generates the elementary group $\Ep(R,\Lambda)$.
\end{lem}

\begin{proof}
If $i\ne\pm j,\pm1$ then $T_{ij}(\mu)=[T_{i1}(\mu),T_{1j}(1)]$.
On the other hand, for $\lambda\in\Lambda$ we have
$T_{-ii}(\lambda)=[T_{-11}(\mu),T_{1i}(1)]T_{-i1}(-\lambda)$.
\end{proof}

\begin{lem}\label{InParabolic}
Let $H$ be a subgroup of $\Sp(A)$, containing $\Ep(K)$ and let
$(R,\Lambda)$ be a form ring, associated with $H$.
Suppose that $g$ commutes with a long root subgroup $X_\gamma(K)$
and $\Ep(K)^g\in H$. Then $g\in N_A(R,\Lambda)$.
\end{lem}

\begin{proof}
Without loss of generality we may assume that
$\gamma=2\eps_1$ is the maximal root. Then $g$ belongs to the standard parabolic subgroup
$P_1$, corresponding to the simple root $\alpha_1=\eps_1-\eps_2$,
in other words, $g_{i1}=g_{-1i}=g_{-11}=0$ for all
$i\ne\pm1$. Moreover, $g_{11}=g_{-1-1}=g'_{11}=g'_{-1-1}=\mu$, where $\mu^2=1$.
Then $g=ab$ for some $b\in U_1(A)$ and $a\in L_1(A)$ ($a$ and $b$
are not necessarily in $H$).
For any $d\in U_1(K)$ the element $d^g$ belongs to  $H\cap U_1(A)$.
By Lemma~\ref{InUniRad}, $d^g\in \Ep(R,\Lambda)$.
If $d=T_{1i}(1)$, then the above inclusion implies that $\mu g_{ij}\in R$
for all $i\ne1$ and all $j\in I$. It follows that $T_{1i}(1)^a$ and
$T_{i1}(1)^a$ belong to $\Ep(R,\Lambda)$. On the other hand,
$a$ commutes with root subgroups $X_{\pm\gamma}(K)$. By the previous lemma
we conclude that $a$ normalizes $\Ep(R,\Lambda)$.

Now, we have $\Ep(K)^b\le\Ep(R,\Lambda)^g\in H$ and by
Lemma~\ref{InUniRad} $b\in\Ep(R,\Lambda)$.
Thus, $g\in N_A(R,\Lambda)$ as required.
\end{proof}

\section{Proof of Theorem~1}\label{2=0}

In this section we assume that $2=0$ in $K$.
Let $G$ be a Chevalley group with a not simply laced root system, e.\,g. $G=\Sp$.
Recall that in this case a short root unipotent element $h$
is called a \emph{small unipotent element}, see~\cite{Gordeev}.
The terminology reflects the fact that the conjugacy class of $h$ is small.
In our settings a small unipotent element is conjugate to $T_{ij}(\mu)$,
where $i\ne\pm j$ and $\mu\in R$.

The following lemma was obtained over a field by Golubchik and Mikhalev
in~\cite{GolMikh}, Gordeev in~\cite{Gordeev} and
Nesterov and Stepanov in~\cite{NestStepF4}.

\begin{lem}\label{identity}
Let $\Phi=B_n,C_n,F_4$ and let $R$ be a ring such that $2=0$.
Let $\alpha$ be a long root and $g\in G(\Phi,R)$.
If $h\in G(R)$ is a small unipotent element, then
$X_\alpha(R)^{h^g}$ commutes with $X_\alpha(R)$.
\end{lem}

\begin{proof}
The identity with constants $[X_\alpha(R)^{h^g},X_\alpha(R)]=\{1\}$ is inherited by
subrings and quotient rings. Any commutative ring with $2=0$ is a quotient of a
polynomial ring over $\mathbb F_2$ which is a subring of a field of
characteristic~$2$.
\end{proof}

The next lemma is the last ingredient for the proof of Theorem~\ref{MAIN}.
It follows from the normal structure of the general unitary group
$\operatorname{GU}_{2n}(R,\Lambda)$ obtained by Bak and Vavilov
at the middle of 1990-s. Since this result has not been published yet,
we give a prove of a very simple special case of it.

\begin{lem}\label{NSofGU}
A normal subgroup $N$ of\/ $\Ep(R,\Lambda)$, containing a root element
$T_{ij}(1)$, coincides with\/ $\Ep(R,\Lambda)$.
\end{lem}

\begin{proof}
First, suppose that $i\ne\pm j$.
Take $k\ne\pm i,\pm j$. Then, $[T_{ij}(1),T_{jk}(\xi)]=T_{ik}(\mu)\in N$.
Since the Weyl group acts transitively on the set of all short roots,
we have $T_{lm}(\mu)\in N$ for all $l\ne\pm m$ and $\mu\in R$.
Further,
$$T_{l,-m}(-\lambda)[T_{lm}(1),T_{m,-m}(\lambda)]=T_{l,-l}(\lambda)\in N$$
for all $\lambda\in\Lambda$ and $l\in I$.

Now, let $i=m$, $j=-m$ and $l\ne\pm m$. Put $\lambda=1$ in the latter commutator identity.
Then, $T_{l,-m}(1)T_{l,-l}(1)\in N$. By transitivity of the Weyl group we know that
$T_{l,-l}(1)\in N$, therefore $T_{l,-m}(1)\in N$. By the first paragraph of the proof,
$N$ contains all the generators of $\Ep(R,\Lambda)$.
\end{proof}

Now we are ready to prove Theorem~\ref{MAIN}.
The idea of the proof is the same as for the main result
of~\cite{StepStandard}.

\begin{proof}[Proof of Theorem~\ref{MAIN}]
Let $(R,\Lambda)$ be the form subring associated with $H$.
Put $h=T_{12}(1)$ and $x=T_{1,-1}(\lambda)$, where $\lambda\in\Lambda$.
Take two arbitrary elements
$a,b$ from $\Ep(R,\Lambda)$ and consider the element $c=x^{h^{agb}}\in H$.
By Lemma~\ref{identity} this element commutes with a long root subgroup
and by Lemma~\ref{InParabolic} $c\in N_A(R)$. Rewrite $c$ in the form
$$
c=\bigl(g^{-1}  (a^{-1}h^{-1}a)g
       (bxb^{-1})
        g^{-1}  (a^{-1}h     a)g\bigr)^b
$$

\noindent
Since $b\in\Ep(R,\Lambda)$, the element $bcb^{-1}$ is in $N_A(R)$.
Fix $a$ and let $b$ and $\lambda$ vary. The subgroup generated by
$bxb^{-1}$ is normal in $\Ep(R,\Lambda)$. By
Lemma~\ref{NSofGU} it must coincide with  $\Ep(R,\Lambda)$. Thus,
$\Ep(R,\Lambda)^{g^{-1} (a^{-1}h^{-1}a)g}\in N_A(R)$,
and by Lemma~\ref{normalizer} $(a^{-1}h^{-1}a)^g\in N_A(R)$.

Again, elements of the form $a^{-1}h^{-1}a$, as $a$ ranges over $\Ep(R,\Lambda)$,
generate a normal subgroup in $\Ep(R,\Lambda)$.
The minimal normal
subgroup of $\Ep(R,\Lambda)$ containing $h$ must be equal to $\Ep(R,\Lambda)$
by Lemma~\ref{NSofGU}.
Therefore, $\Ep(R,\Lambda)^g\in N_A(R)$. By Lemma~\ref{normalizer} one
has $g\in N_A(R)$, which completes the proof.
\end{proof}

\providecommand{\bysame}{\leavevmode\hbox to3em{\hrulefill}\thinspace}
\providecommand{\MR}{\relax\ifhmode\unskip\space\fi MR }
\providecommand{\MRhref}[2]{%
  \href{http://www.ams.org/mathscinet-getitem?mr=#1}{#2}
}
\providecommand{\href}[2]{#2}

\end{document}